\documentclass[a4paper,10pt]{article}

\usepackage[T1]{fontenc} 
\usepackage{amsmath,amsthm}
\usepackage[dvips]{graphicx}
\usepackage{amsfonts,amssymb}
\usepackage{geometry}
\usepackage{hyperref}
\usepackage{stmaryrd}
\usepackage{fancyhdr}
\usepackage{url}

\newtheorem{prop}{Proposition}[section]
\newtheorem{LM}{Lemma}[section]
\newtheorem{thm}{Theorem}[section]

\newtheorem{cor}{Corollary}[section]

\title {$\mathbf{Distinguished \ principal \ series \ representations}$ $\mathbf{for \ GL_n \ over \ a \ p-adic \ field}$}
\author{Nadir Matringe}

\begin{document}
\maketitle

\section{Introduction}

For $K/F$ a quadratic extension of p-adic fields,  let $\sigma$ be the
 conjugation relative to this extension, and $\eta _{K/F}$ be the
 character of $F^*$ with kernel norms of $K^*$.\\
If $\pi$ is a smooth irreducible representation of $GL(n,K)$, and $\chi$ a character of $F^*$, the
 dimension of the space of linear forms on its space, which transform by $\chi$ under $GL(n,F)$ (with respect to the action $[(L,g)\mapsto L\circ\pi(g)]$) , is known to be at
 most one (Proposition 11, \cite{F}).
One says that $\pi$ is $\chi$-distinguished if this dimension is one, and says that $\pi$ is distinguished if it is $1$-distinguished.\\
In this article, we give a description of distinguished principal
 series representations of $GL(n,K)$.\\
The result (Theorem \ref{dist}) is that the the irreducible distinguished representations of the principal series of
$GL(n,K)$ are (up to isomorphism) those unitarily induced from a character $\chi=(\chi_1,...,\chi_n)$ of the maximal torus of diagonal matrices, such that there exists $r \leq n/2$, for which $ \chi
 _{i+1}^{\sigma} = {{\chi _i}} ^{-1} $ for $i=1,3,..,2r-1$, and ${\chi _i} _{|F^*} =1$
 for $ i > 2r$.
For the quadratic extension $\mathbb{C}/  \mathbb{R}$, it is known
 (cf.\cite{P}) that the analogous result is true for tempered representations.\\
For $n\geq 3$, this gives a counter-example (Corollary \ref{contre}) to a conjecture of Jacquet (Conjecture 1 in \cite{A}). This conjecture states that an irreducible representation $\pi$ of $GL(n,K)$ with central character trivial on $F^*$ is isomorphic to $\check{\pi}^{\sigma}$ if and only if it is distinguished or $\eta _{K/F}$-distinguished (where $\eta _{K/F}$ is the character of order 2 of $F^*$, attached by local class field theory to the extension $K/F$). For discrete series representations, the conjecture is verified, it was proved in \cite{K}.\\
Unitary irreducible distinguished principal series representations of $GL(2,K)$ were described in \cite{H}, and the general case of distinguished irreducible principal series representations of $GL(2,K)$ was treated in \cite{FH}. We use this occasion to give a different proof of the result for $GL(2,K)$ than the one in \cite{FH}.
To do this, in Theorems \ref{gamma1} and \ref{gamma2}, we extend a
 criterion of Hakim (th.4.1, \cite{H}) characterising smooth unitary 
irreducible distinguished representations of $GL(2,K)$ in terms of $\gamma$ factors at $1/2$,
 to all smooth irreducible distinguished representations of $GL(2,K)$.\\

\section{Preliminaries}

Let $\phi$ be a group automorphism, and $x$ an element of the group, we
 sometimes note $x^{\phi}$ instead of $\phi(x)$, and $x^{-\phi}$ the
 inverse of $x^{\phi}$.
If $\phi = x \mapsto h^{-1}xh$ for $h$ in the group, then $x^{h}$
 designs $x^{\phi}$.\\

Let $G$ be a locally compact totally disconnected group, $H$ a closed
 subgroup of $G$.\\
We note $\Delta_{G}$ the module of $G$, given by the relation $d_G(gx)=\Delta_{G}(g)d_G(x)$, for a right Haar measure $d_G$ on $G$.\\
Let $X$ be a locally closed subspace of $G$, with $H.X \subset X$.
If $V$ is a complex vector space, we note $D(X,V)$ the space of smooth
 $V$-valued functions on $X$ with compact support (if $V=\mathbb{C}$, we
 simply note it $D(X)$).\\
Let $\rho$ be a a smooth representation of $H$ in a complex vector
 space $V_{\rho}$, we note $D(H \backslash X, \rho, V_{\rho})$ the space of
 smooth $V_{\rho}$-valued functions $f$ on $X$, with compact support
 modulo $H$, which verify $f(hx)=\rho(h)
 f(x)$ for $h \in H$ and $x \in X$ (if $\rho$ is a character, we note
 it $D(H \backslash X, \rho)$).\\ 
We note $ind _H^G (\rho)$ the representation by right translations of
 $G$ in \\$D(H \backslash G, (\Delta _G /\Delta _H)^{1/2} \rho,
 V_{\rho})$.\\ 

Let $F$ be a non archimedean local field of characteristic zero, and
 $K$ a quadratic extension of $F$. We have $K = F(\delta)$ with $\delta ^2$ in $F^*$.\\
We note $| \ |_K$ and $| \ |_F$ the modules of $K$ and $F$
 respectively.\\
We note $\sigma$ the non trivial element of the Galois group $G(K/F)$
 of $K$ over $F$, and we use the same letter to design for its action on $M_n(K)$.\\
We note $N_{K/F}$ the norm of the extension $K/F$ and we note $\eta _{K/F}$ the nontrivial character of $F^*$ which is trivial on $N_{K/F}(K^*)$.\\
Whenever $G$ is an algebraic group defined over $F$, we note $G(K)$ its
 $K$-points and $G(F)$ its $F$-points.\\

The group $GL(n)$ will be noted $G_n$, its standard Borel subgroup will
 be noted $B_n$, its unipotent radical $U_n$, and the standard maximal split torus of diagonal matrices $T_n$.\\

We note $S$ the space of matrices $M$ in $G_n(K)$ satisfying $MM^{\sigma}=1$.\\

Everything in this paragraph is more or less contained in \cite{F1}, we give detailed proofs here for convenience of the reader.\\

\begin{prop}{ (\cite{S}, ch.10, prop.3)}\label{cohom}
 
We have a homeomorphism between $G_n(K)/G_n(F)$ and $S$ given by the
 map $S_n: g \mapsto g^{\sigma}g^{-1}$.

\end{prop}

\begin{prop}
\label{order2}

For its natural action on $S$, each orbit of $B_n (K)$ contains one and
 only one element of $\mathfrak{S} _n$ of order 2 or 1.

\end{prop}

\begin{proof}
We begin with the following:

\begin{LM}
\label{nilp1}

Let $w$ be an element of $\mathfrak{S} _n \subset G_n(K)$ of order at most $2$.
 
Let $\theta '$ be the involution of $T_n(K)$ given by $t \mapsto w^{-1}
 t^{\sigma} w$, then any $t\in T_n(K)$ with $t \theta '(t)=1$ is of the
 form $a/ \theta '(a)$ for some $a\in T_n(K)$. 

\end{LM}

\begin{proof}[Proof of Lemma \ref{nilp1}.]
There exists $r\leq n/2$ such that up to conjugacy, $w$ is\\
 $(1,2)(3,4)...(2r-1,2r)$.\\

We write $t=\left (\begin{array}{cccccccc}
z_1 &  &  &  &  &  &  & \\
    & {z_1}' &  &  &  &  &   \\
    &  & \ddots &  &  &  &  &  \\  
    &  &     & z_r &  &  &  &  \\
    &  &     &   & {z_r}'&  &  &  \\
    &  &     &   &               & z_{2r+1} &  &    \\
    &  &     &   &               &     & \ddots&    \\
    &  &     &   &               &     &       & z_n \\  
 \end{array}\right)$, hence for $i\leq r$, we have $z_i \sigma (z_i
 ')=1$, and $z_j \sigma(z_j)=1 $ for $j \geq 2r+1$.\\
 Hilbert's Theorem 90 asserts that each $z_j,j \geq 2r+1$ is of the
 form $u_{j-2r}/ \sigma(u_{j-2r})$, for some $u_{j-2r} \in K^*$.\\
We then take $a=\left (\begin{array}{cccccccc}
z_1 &  &  &  &  &  &  & \\
    & 1 &  &  &  &  &   \\
    &  & \ddots &  &  &  &  &  \\  
    &  &     & z_r &  &  &  &  \\
    &  &     &   & 1 &  &  &  \\
    &  &     &   &               & u_1 &  &    \\
    &  &     &   &               &     & \ddots&    \\
    &  &     &   &               &     &       & u_{n-2r} \\ 
\end{array}\right)$.\end{proof}

\begin{LM}
\label{nilp2}

Let $N$ be an algebraic connected unipotent group over $K$. Let $\theta$ be an
 involutive automorphism of $N(K)$. If $x \in N(K)$, verifies $x
 \theta(x)= 1_N$ then, there is $a \in N$ such that $x = \theta (a ^{-1}) a$.

\end{LM}

\begin{proof}[Proof of Lemma \ref{nilp2}:]
The group $N(K)$ has a composition series $ 1_N = N_0 \subset N_1 \subset ...
 N_{n-1} \subset N_n=N(K)$, such that each quotient $N_{i+1}/N_i$ is isomorphic
 to $(K,+)$, and each commutator subgroup $\left[ N, N_{i+1} \right]$ is a
 subgroup of $N_i$.\\
Now we prove the Lemma by induction on $n$:\\
If $n=1$, then $N(K)$ is isomorphic to $(K,+)$, one concludes taking $a =
 x/2$.\\
$n \mapsto n+1:$\\
suppose the Lemma is true for every $N(K)$ of length $n$.\\
Let $N(K)$ be of length $n+1$, we note $\bar{x}$ the class of $x$ in
 $N(K)/N_1$.\\
By induction hypothesis, one gets that there exists an element in $h
 \in N_1$, and an element $u$ in $N(K)$ such that $x= \theta(u^{-1})u h$.\\
Here $h$ lies in the center of $N(K)$, because $\left[ N(K), N_{1} \right] =
 1_N$.\\
 As $x \theta (x) = 1$, we get $h \theta(h) =1$. By induction
 hypothesis again, we get $h= \theta (b^{-1}) b$ for $b \in N_1$. We then take $a
 = ub$. \end{proof}   

We get back to the proof of the Proposition \ref{order2}.\\

For $w$ in $\mathfrak{S} _n$, one notes $U_w$ the subgroup of $U_n$
 generated by the elementary subgroups $U_{\alpha}$, with $\alpha$
 positive, and $w \alpha$ negative, and ${U_w}'$ the subgroup of $U_n$ generated
 by the elementary subgroups $U_{\alpha}$, with $\alpha$ positive, and
 $w \alpha$ positive. Then $U_n= {U_w}' U_w$.\\ 
Let $s$ be in $S$. According to Bruhat's decomposition, there is $w$ in
 $\mathfrak{S} _n$, and $a$ in $T_n(K)$, $n_1$ in $U_n(K)$ and
 ${n_2}^+$ in $U_w$, such that $s=n_1 a  w {n_2}^+$, with unicity of the
 decomposition.\\
Then $s=s^{-\sigma}={{n_2}^+}^{-\sigma}  w^{-1}{a}^{-\sigma}
 {n_1}^{-\sigma}$.\\
Thus we have ${a w} = (a w )^{-\sigma}$, i.e. $w^2=1$ and $a^w=
 a^{-\sigma}$.\\
Now we write ${n_1}^{-\sigma}= u^{-}u^{+}$ with $u^{-} \in {U_w}'$ and
 $u^{+} \in {U_w}$, comparing $s$ and $s^{-\sigma}$, $u^{+}$ must be
 equal to ${n_2}^{+}$.\\
Hence $s= n_1 a  w {u^{-}}^{-1}{n_1}^{-\sigma} $, thus we suppose $s= a
  w n$, with $n$ in $U_w '$.\\
From $s=s^{-\sigma}$, one has the relation $ a  w n (aw)^{-1}=
 n^{-\sigma}$, applying $\sigma$ on each side, this becomes $( aw)^{-1}
 n^{\sigma} a  w = n^{-1}$.\\
But $\theta: u \mapsto ( aw)^{-1} u^{\sigma} a  w $ is an involutive
 automorphism of $U_w '$, hence from Lemma \ref{nilp2}, there is $u'$ in
 $U_w '$ such that $n = \theta (u^{-1}) u$.\\
This gives $ s= u^{-\sigma} a  w u$, so that we suppose $s= aw$. Again
 $w a ^\sigma w = a^{-1}$, and applying Lemma \ref{nilp1} to $\theta ' :
 x \mapsto w x ^\sigma w$, we deduce that $a$ is of the form $ y \theta
 '(y^{-1})$, and $s=y w y^{-\sigma}$. \end{proof}

Let $u$ be the element $\left(\begin{array}{cccc}
1 & -\delta \\
1 &  \delta 
 \end{array}\right )$ of $M_2(K)$; one has $S_2(u)= \left (\begin{array}{cccc}
0 & 1 \\
1 &  0 
 \end{array}\right )$ (cf. Proposition \ref{cohom}).\\
We notice for further use (cf. proof of Proposition 3.1), that if we note $\tilde{T}$ the subgroup $\left\lbrace\left (\begin{array}{cc}
z & 0 \\
0 & {z}^{\sigma}\\
 \end{array}\right )\in G_2(K)| z \in K^*\right\rbrace $,
 then $u^{-1}\tilde{T}u= T = \left\lbrace \!\!\left(\begin{array}{cc}
x & \Delta y \\
y & x\\
 \end{array}\right )\!\!\!\in \!\!G_2(F)| x,y \in F \right\rbrace $.\\
\\

For $r \leq n/2$, one notes $U_r$ the $n\times n$ matrix given by the
 following block decomposition: $\left (\begin{array}{ccccc}
u &  & & & \\
 & \ddots & & &\\
 &        & u &  \\
 &        &   & I_{n-2r} \\ 
 \end{array}\right )$

If $w$ is an element of $\mathfrak{S}_n$ naturally injected in
 $G_n(K)$, one notes $U_r ^{w}= w^{-1} U_r w$.

\begin{cor}\label{2bclasses}

The elements $U_r ^{w}$ for $ 0 \leq r \leq n/2$, and $w \in \mathfrak{S}_n$ give
 a complete set of representatives of classes of $  B_n(K)\backslash
 G_n(K)/G_n(F)$.

\end{cor}

Let $G_n= \coprod _{w \in \mathfrak{S} _n} B_n w B_n$ be the Bruhat
 decomposition of $G_n$. We call a double-class $BwB$ a Bruhat cell.  

\begin{LM}
 
One can order the Bruhat cells $C_1$, $C_2$,..., $C_{n!}$ so that for
 every $1 \leq i \leq n!$, the cell $C_i$ is closed in $G_n - \coprod
 _{k=1} ^{i-1} C_i$.

\end{LM}

\begin{proof}
Choose $C_1 = B_n$. It is closed in $G_n$.
Now let $w_2$ be an element of $\mathfrak{S}_n -{Id}$, with minimal length. Then
 from 8.5.5. of \cite{Sp}, one has that the Bruhat cell $B {w_2} B$ is
 closed in $G_n - {B_n}$ with respect to the Zariski topology, hence for the
 p-adic topology, we call it $C_2$. We conclude by repeating this
 process.\end{proof} 

\begin{cor}
\label{closedclasses}
 
One can order the classes $A_1$,..., $A_t$ of $B_n(K)\backslash
 G_n(K)/G_n(F)$, so that $A_i$ is closed in $G_n(K) - \coprod _{k=1} ^{i-1}
 A_i$.

\end{cor}
\begin{proof} From the proof of Proposition \ref{order2}, we know that if $C$ is a Bruhat cell of $G_n$, then $S_n\cap C$ is either empty, or it corresponds through the homemorphism $S_n$ to a class $A$ of $B_n(K)\backslash
 G_n(K)/G_n(F)$. The conclusion follows the preceeding Lemma.
 \end{proof}

\begin{cor}
\label{localclosed}
 
Each $A_i$ is locally closed in $G_n(K)$ for the Zariski topology.

\end{cor}

We will also need the following Lemma:

\begin{LM}
 
Let $G$, $H$,$X$, and $(\rho, V_{\rho})$ be as in the beginning of the section, the map
 $\Phi$ from $D(X)\otimes V_{\rho}$ to $D(H \backslash X, \rho, V_{\rho})$ defined by $\Phi: f\otimes v 
 \mapsto (x \mapsto \int _H f(hx) {\rho} (h^{-1}) v dh)$ is surjective.

\end{LM}

\begin{proof}
Let $v\in V_{\rho}$, $U$ an open subset of $G$ that intersects $X$,
 small enough for $h \mapsto \rho(h)v$ to be trivial on $H \cap U
 U^{-1}$.\\
Let $f'$ be the function with support in $H(X\cap U)$ defined by $hx
 \mapsto \rho(h)v$.\\
Such functions generate $D(H \backslash X,\rho ,V_{\rho})$ as a vector
 space.\\
Now let $f$ be the function of $D( X, V_{\rho})$ defined by $x \mapsto
 1_{U\cap X}(x) v$, then $\Phi(f)$ is a multiple of $f'$.\\
But for $x$ in $U\cap X$, $\Phi(f)(x) = \int _H \rho(h^{-1}) f(hx) dh $
 because $h \mapsto \rho(h)v$ is trivial on $H \cap U U^{-1}$, plus $h
 \mapsto f(hx)$ is a positive function that multiplies $v$, and
 $f(x)=V$, so $F(f)(x)$ is $v$ multiplied by a strictly positive scalar.\end{proof}

\begin{cor}\label{res}
 
Let $Y$ be a closed subset of $X$, $H$-stable, then the restriction map
 from $D(H \backslash X, \rho , V_{\rho})$ to $D(H \backslash Y, \rho,
 V_{\rho})$ is surjective.

\end{cor}

\begin{proof}
This is a consequence of the known surjectivity of the restriction map
 from $D(X)$ to $D(Y)$, which implies the surjectivity of the
 restriction from $D(X, V_{\rho})$ to $D(Y, V_{\rho})$ and of the commutativity of
 the diagram:\\

$\begin{array}{ccc}
D(X) & \rightarrow & D(Y)\\
\downarrow \Phi &  & \downarrow \Phi\\
D(H \backslash X, \rho)& \rightarrow & D(H \backslash Y, \rho)
 \end{array}$. \end{proof}

\section{Distinguished principal series}

If $\pi$ is a smooth representation of $G_n(K)$ of space $V_{\pi}$, and $\chi$ is a character of $F^*$, we say that $\pi$ is $\chi$-distinguished if there exists on $V_{\pi}$ a nonzero linear form $L$ such that $L(\pi(g)v)=\chi(det(g))L(v)$ whenever $g$ is in $G_n(F)$ and $v$ belongs to $V_{\pi}$. If $\chi$ is trivial, we simply say that $\pi$ is distinguished.\\

We first recall the following:

\begin{thm}{(\cite{F}, Proposition 12)}
\label{autodual}
 
Let $\pi$ be a smooth irreducible distinguished representation of $G_n(K)$, then $\pi
 ^{\sigma} \simeq \check{\pi}$. 

\end{thm}

Let $\chi _1, ... , \chi _n$ be $n$ characters of $K^*$, with none of
 their quotients equal to $| \ |_K$.
We note $\chi$ the character of $B_n(K)$ defined by $\chi \left
 (\begin{array}{ccc}
b_1 & \star & \star \\
 & \ddots & \star \\
 &        & b_n 
 \end{array}\right ) = \chi _1(b_1)... \chi _n (b_n)$.\\
 We note $\pi( \chi)$ the representation of $G_n(K)$
 by right translation on the space of functions $D(B_n (K)\backslash
 G_n(K), {\Delta _{B_n}}^{-1/2} \chi)$.
This representation is smooth, irreducible and called the principal
 series attached to $\chi$.\\
If $\pi$ is a smooth representation of $G_n(K)$, we note $\check{pi}$ its smooth contragredient.\\

We will need the following Lemma:

\begin{LM}(Proposition 26 in \cite{F1})
\label{induced}
Let ${\bar{m}}=(m_1,\dots,m_l)$ be a partition of a positive integer $m$,
 let $P_{\bar{m}}$ be the corresponding standard parabolic subgroup,
 and for each $1\leq i\leq l$, let $\pi _i$ be a smooth distinguished representation of $G_{m_i}(K)$,
 then $\pi _1 \times \dots \times \pi _l=
 ind_{P_{\bar{m}}(K)}^{G_m(K)}({\Delta _{P_{\bar{m}}(K)} ^{-1/2}}(\pi _1 \otimes \dots \otimes \pi
 _l)) $ is distinguished.

\end{LM}

We now come to the principal results:

\begin{prop}
 
Let $\chi=(\chi_1,\dots,\chi_n)$ be a character of $T_n(K)$, suppose that the principal series representation $\pi(\chi)$ is distinguished, there exists a
 re-ordering of the ${\chi _i}$'s, and $r \leq n/2$, such that $ \chi
 _{i+1}^{\sigma} = {{\chi _i}} ^{-1} $ for $i=1,3,..,2r-1$, and that ${\chi _i} _{|F^*} =1$
 for $ i > 2r$.

\end{prop}

\begin{proof} We write $B=B_n(K)$, $G=G_n(K)$.
We have from Corollary \ref{closedclasses} and \ref{res} the
 following exact sequence of smooth $G_n(F)$-modules: $$D(B \backslash {G}
 - {A_1}, {\Delta}_{B}^{-1/2} \chi) \hookrightarrow
 D(B\backslash G, {\Delta}_{B}^{-1/2} \chi)\rightarrow D(B
 \backslash A_1 , {\Delta}_B^{-1/2} \chi).$$\\
Hence there is a non zero distinguished linear form either on $D(B
 \backslash A_1, {\Delta}_{B}^{-1/2} \chi)$, or on\\ $D(B \backslash G -
 A_1, {\Delta}_{B}^{-1/2} \chi)$.\\
In the second case we have the exact sequence $$D(B \backslash G -
 A_1\sqcup A_2 , {\Delta}_{B}^{-1/2} \chi) \hookrightarrow D(B
 \backslash G - A_1, {\Delta}_{B}^{-1/2} \chi) \rightarrow D(B
 \backslash A_2, {\Delta}_{B}^{-1/2} \chi).$$\\
Repeating the process, we deduce the existence of a non zero
 distinguished linar form on one of the spaces $D(B \backslash A_i,
{\Delta}_{B}^{-1/2} \chi)$.\\
From Corollary \ref{2bclasses}, we choose $w$ in $S_n$ and $r \leq n/2$ such that $A_i = B
 U_r^{w} G_n(F)$.
The application $f \mapsto \left[ x \mapsto f( U_r^w x
 )\right]$ gives an isomorphism of $G_n(F)$-modules between $D(B \backslash A_i,
 {\Delta}_{B}^{-1/2} \chi)$ and $D( U_r^{-w} B U_r^w \cap
 G_n(F)\backslash G_n(F), {\Delta}' \chi ' )$ where ${\Delta}' (x) =
 {\Delta}_{B}^{-1/2} ( U_r^w x U_r^{-w} )$ and $\chi '(x) = \chi(
 U_r^w x U_r^{-w})$.\\ 
Now there exists a nonzero $G_n(F)$-invariant linear form on\\ $D( U_r^{-w} B U_r^w
 \cap G_n(F)\backslash G_n(F), {\Delta}' \chi ' )$ if and only if
 ${\Delta}' \chi '$ is equal to the inverse of the module of $U_r^{-w} B
 U_r^w \cap G_n(F)$ (cf.\cite{BH}, ch.1, prop.3.4).
From this we deduce that $\chi '$ is positive on $U_r^{-w} B
 U_r^w \cap G_n(F)$ or equivalently $\chi$ is positive on $B \cap U_r^w
  G_n(F) U_r^{-w}$.\\
Let $\bar{T_r}$ be the $F$-torus of matrices of the form $$\left
 (\begin{array}{cccccccc}
z_1 &  &  &  &  &  &  & \\
    & {z_1}^{\sigma} &  &  &  &  &   \\
    &  & \ddots &  &  &  &  &  \\  
    &  &     & z_r &  &  &  &  \\
    &  &     &   & {z_r}^{\sigma}&  &  &  \\
    &  &     &   &               & x_1 &  &    \\
    &  &     &   &               &     & \ddots&    \\
    &  &     &   &               &     &       & x_t \\  
 \end{array}\right)$$ with $2r+t=n$, $z_i \in K^*$, $x_i \in F^*$, then one
 has $\bar{T_r} ^{w} \subset B \cap U_r^w  G_n(F) U_r^{-w}$,
 so that $\chi$ must be positive on $\bar{T_r} ^{w}$.\\
 We $\!$remark $\!$that $\!$if $\chi$ is $\!$unitary, $\!$then $\chi$ is $\!$trivial $\!$on
 $\bar{T_r} ^{w}$, and $\pi(\chi)$ is $\!$of $\!$the desired form.\\
For the general case, we deduce from Theorem \ref{autodual}, that there
 exists three integers $p\geq 0,q\geq 0,s\geq 0$ such that up to
 reordering, we have $\chi_{2i}=\chi_{2i-1}^{-\sigma}$ for $1 \leq i \leq p$,
 we have ${\chi _{2p+k}}_{|F^*}=1$ for $1 \leq k \leq q$ and these $\chi
 _{2p+k}$'s are different (so that $\chi _{2p+k}\neq \chi _{2p+k'}^{-\sigma}$ for $k\neq k'$), and ${{\chi _{2p+q+j}}}_{|F^*}= \eta _{K/F}$ for $1\leq j \leq s$, these ${\chi _{2p+q+j}}$'s being different.\\
We note $\mu _k = {\chi _{2p+k}}$ for $ q \geq k \geq 1$, and $\nu _k' =
 \chi _{2p+q+k'}$  for $s \geq k' \geq 1$.\\
We show that if such a character $\chi$ is positive on a conjugate of
 $\bar{T_r}$ by an element of $S_n$, then $s=0$.\\ 
%For $n=2$:\\
%suppose $\chi _1_{|F^*} = \eta _{K/F}$, then , then necessarily $r=2$,
% and $\chi$ must be positive on elements of the form $\left
 %(\begin{array}{cccc}
%z &  \\
 %& \bar{z}
 %\end{array}\right )$, which implies $\chi_2=\chi_1^{-\sigma} {| \ |
% _K}^t$ for some real $t$, but because of the form of $\chi$, one must
 %have $\chi_2=\chi_1^{-\sigma}$.\\  
 %$n \rightarrow n+1$:
Supose $\nu _1$ appears, then either $\nu _1$ is positive on $F^*$, but
 that is not possible, or it is coupled with another ${\chi _i}$, and
 $(\nu _1, {\chi _i})$ is positive on elements $(z, z^{\sigma})$, for $z$
 in $K^*$.\\
Suppose ${\chi _i} = \nu _j$ for some $j \neq 1$, then $(\nu _1, \chi
 _i)$ is unitary, so it must be trivial on couples $(z, z^{\sigma})$, which
 implies $\nu _1 = {\nu _j}^{-\sigma}=\nu _j$, which is absurd.\\
The character ${\chi _i}$ cannot be of the form $\mu _j$, because it
 would imply ${\nu _1} _{|F^*} =1$.\\
The last case is $i \leq 2p$, then ${\nu _1}^{-\sigma}=\nu _1$ must be
 the unitary part of ${\chi _i}$ because of the positivity of $(\nu _1,
 {\chi _i})$ on the couples $(z, z^{\sigma})$.\\
But ${{\chi _i}}^{-\sigma}$ also appears and is not trivial on $F^*$, hence must be coupled with
 another character $\chi _j$ with $j \leq 2p$ and $j\neq i$, such that $({\chi _i}^{-\sigma},\chi _j)$ is positive
 on the elements $(z, z^{\sigma})$, for $z$ in $K^*$, which implies that
 $\chi _j$ has unitary part ${\nu _1}^{-\sigma}=\nu _1$.
The character $\chi _j$ cannot be a $\mu _k$ because of its unitary part.\\
 If it is a $\chi _k$ with $k \leq 2p$, we consider again ${\chi _k}
 ^{-\sigma}$ .\\
 But repeating the process lengthily enough, we can suppose that $\chi
 _j$ is of the form $\nu _k$, for $k \neq 1$. Taking unitary parts, we
 see that $\nu _k= {\nu _1}^{-\sigma}= \nu _1$, which is in contradiction
 with the fact that all $\nu _i$'s are different.
We conclude that $s=0$.  \end{proof}

\begin{thm}
\label{dist}
 
Let $\chi=(\chi_1,\dots,\chi_n)$ be a character of $T_n(K)$, the principal series representation $\pi(\chi)$ is distinguished if and only if
 there exists $r \leq n/2$, such that $ \chi _{i+1}^{\sigma} = {{\chi _i}}
 ^{-1} $ for $i=1,3,..,2r-1$, and that ${\chi _i} _{|F^*} =1$ for $ i > 2r$.

\end{thm}

\begin{proof}
There is one implication left.\\
Suppose $\chi$ is of the desired form, then $\pi(\chi)$ is
 parabolically (unitarily) induced from representations of the type $\pi({\chi _i},
 {{\chi _i}}^{-\sigma})$ of $G_2(K)$, and distinguished characters of
 $K^*$.\\
Hence, because of Lemma \ref{induced} the Theorem will be proved if we
 know that the representations $\pi({\chi _i}, {{\chi _i}}^{-\sigma})$ are
 distinguished, but this is Corollary \ref{cordist2} of the next
 paragraph. \end{proof} 

This gives a counter-example to a conjecture of Jacquet (conjecture 1 in \cite{A}), asserting that if an irreducible admissible representation $\pi$ of $G_n(K)$ verifies that $\check{\pi}$ is isomorphic to $\pi ^{\sigma}$, then it is distinguished if $n$ is odd, and it is distinguished or $\eta_{K/F}$-distinguished if $n$ is even.

\begin{cor}\label{contre}
 
For $n \geq 3$, there exist smooth irreducible representations $\pi$ of $G_n(K)$, with central character trivial on $F^*$, that are neither distinguished, nor\\ $\eta_{K/F}$-distinguished, but verify that
$\check{\pi}$ is isomorphic to $\pi ^{\sigma}$.

\end{cor}

\begin{proof}
Take $\chi_1,\dots,\chi_n$, all different, such that ${\chi_1}_{|F^*}={\chi_2}_{|F^*}=\eta_{K/F}$, and ${\chi_j}_{|F^*}=1$ for $3\leq j \leq n$. Because each $\chi_i$ has trivial restriction to $N_{K/F}(K^*)$, it is equal to $\chi_i^{-\sigma}$, hence $\check{\pi}$ is isomorphic to $\pi^{\sigma}$. Another consequence is that if $k$ and $l$ are two different integers between $1$ and $n$, then $\chi_k \neq \chi_l^{-\sigma}$, because we supposed the $\chi_i$'s all different.\\
Then it follows from Theorem \ref{dist} that $\pi=\pi(\chi_1,\dots,\chi_n)$ is neither distinguished, nor $\eta_{K/F}$-distinguished, but clearly, the central character of $\pi$ is trivial on $F^*$ and $\check{\pi}$ is isomorphic to ${\pi}^{\sigma}$. \end{proof}

\section{Distinction and gamma factors for $GL(2)$}

As said in the introduction, in this section we generalize to smooth infinite dimensional irreducible representations of $G_2(K)$ a criterion of Hakim (cf. \cite{H}, Theorem 4.1) characterising smooth unitary irreducible distinguished representations of $G_2(K)$. In proof of Theorem 4.1 of \cite{H}, Hakim deals with unitary representations so that the integrals of Kirillov functions on $F^*$ with respect to a Haar measure of $F^*$ converge. We skip the convergence problems using Proposition 2.9 of chapter 1 of \cite{JL}. \\

We note $M(K)$ the mirabolic subgroup of $G_2(K)$ of matrices of the form $\left(\begin{array}{lll} a & x \\ 0& 1 \\ \end{array}
 \right)$ with $a$ in $K^*$ and $x$ in $K$, and $M(F)$ its intersection with $G_2(F)$.
We note $w$ the matrix $\left(\begin{array}{lll} 0 & -1 \\ 1& 0 \\ \end{array}
 \right)$.\\
Let $\pi$ be a smooth infinite dimensional irreducible representation of
 $G_2(K)$, it is known that it is generic (cf.\cite{Z} for example).
Let $K(\pi, \psi)$ be its Kirillov model corresponding to $\psi$ (\cite{JL}, th. 2.13), it contains the subspace
 $D(K^*)$ of functions with compact support on the group $K^*$. 
If $\phi$ belongs to $K(\pi, \psi)$, and $x$ belongs to $K$, then $\phi
 - \pi\left(\begin{array}{lll} 1 & x \\ 0& 1 \\ \end{array}
 \right)\phi$ belongs to $D(K^*)$ (\cite{JL}, prop.2.9, ch.1), from this follows that $K(\pi, \psi)= D(K^*) + \pi(w)D(K^*)$.\\

We now recall a consequence of the functional equation at 1/2 for Kirillov
 representations (cf. \cite{B}, section 4.7).\\
 Forall $\phi$ in  $K ( \pi , \psi)$ and $\chi$ character of
 $K^*$, we have whenever both sides converge absolutely:

\begin{eqnarray}\label{eq}
 \begin{array}{r}
  \int_{K^*}  \pi (w) \phi (x) (c_{\pi} \chi)^{-1} (x)   d^* x =
  \gamma( \pi \otimes \chi , \psi)  \int_{K^*} \phi (x) \chi (x)  d^* x 
 \end{array}
\end{eqnarray}

where $d^*x$ is a Haar measure on $K^*$, and $c_{\pi}$ is the central
 character of $\pi$.\\

\begin{thm}
\label{gamma1}
Let $\pi$ be a smooth irreducible representation of $G_2(K)$ of infinite dimension
 with central character trivial on $F^*$, and $\psi$ a
 nontrivial character of $K$ trivial on $F$. If $\gamma (\pi \otimes \chi ,
 \psi)=1 $ for every character $\chi$ of $K^*$ trivial on $F^*$, then
  $\pi$ is distinguished. 

\end{thm}

\begin{proof}
In fact, using a Fourier inversion in functional equation \ref{eq} and the
 change of variable $x \mapsto x^{-1}$, we deduce that forall $\phi$ in
 $D (K^*) \cap \pi (w) D (K^*)$, we have
 $$ c_{\pi} (x) \int_{F^*} \pi (w) \phi (t x^{-1}) d^*t = \int_{F^*}
  \phi (t x) d^*t $$ ($d^*t$ is a Haar measure on $F^*$) which
for $x=1$ gives 
$$ \int_{F^*} \pi (w) \phi (t ) d^*t = \int_{F^*}  \phi (t ) d^*t. $$

Now we define on $K(\pi,\psi)$ a linear form $\lambda$ by:

$$ \lambda ( \phi _1 + \pi (w) \phi _2 ) = \int_{F^*}  \phi _1 (t )
 d^*t + \int_{F^*}  \phi _2 (t ) d^*t $$ 

\noindent for $\phi_1$ and $\phi_2$ in $D(K^*)$, which is well defined because of the previous equality and the fact that $K(\pi,\psi)=D(K^*)+\pi(w)D(K^*)$.\\

It is clear that $\lambda$ is $w$-invariant.
As the central character of $\pi$ is trivial on $F^*$, $\lambda$ is also $F^*$-invariant.
Because $GL_2 (F)$ is generated by $M(F)$, its center, and $w$, it
 remains to show that $\lambda$ is $M(F)$-invariant.\\
 Since $\psi$ is trivial on $F$, one has if $\phi \in D(K^*)$ and $m \in M(F)$ the equality $\lambda (\pi
 (m) \phi) = \lambda (\phi)$ .\\
Now if $\phi=\pi(w)\phi _2 \in \pi(w) D(K^*)$, and if $a$ belongs to $F^*$, then $\pi \left(\begin{array}{lll} a & 0 \\ 0&
 1 \\ \end{array} \right) \pi(w) \phi _2= \pi(w) \pi
 \left(\begin{array}{lll} 1 & 0 \\ 0& a \\ \end{array} \right) \phi_2 = \pi(w) \pi
 \left(\begin{array}{lll} a^{-1} & 0 \\ 0& 1 \\ \end{array} \right) \phi_2$
 because the central character of $\pi$ is trivial on $F^*$, and $\lambda(\pi
 \left(\begin{array}{lll} a & 0 \\ 0& 1 \\ \end{array} \right)\phi)=
 \lambda(\phi)$.\\
If $x\in F$, then $\pi\left(\begin{array}{lll} 1 & x \\ 0& 1 \\
 \end{array} \right)\phi-\phi$ is a function in $D(K^*)$, which vanishes on
 $F^*$, hence $\lambda \pi(\left(\begin{array}{lll} 1 & x \\ 0& 1 \\
 \end{array} \right)\phi-\phi)=0$.\\

Eventually $\lambda$ is $M(F)$-invariant, hence $G_2(F)$-invariant, it is clear that its restriction to $D(K^*)$ is non zero. \end{proof}

\begin{cor}
\label{cordist2}
 
Let $\mu$ be a character of $K^*$, then $\pi(\mu, \mu ^{-\sigma})$ is
 distinguished.

\end{cor}

\begin{proof}
indeed, first we notice that the central character $\mu \mu^{-\sigma}$
 of $\pi(\mu, \mu ^{-\sigma})$ is trivial on $F^*$.\\
Now let $\chi$ be a character of $K^* / F^*$, then $\gamma(\pi(\mu, \mu
 ^{-\sigma}) \otimes \chi,\psi)=\gamma(\mu\chi,\psi)\gamma(\mu
 ^{-\sigma}\chi,\psi)=\gamma(\mu \chi,\psi)\gamma(\mu^{-1}\chi^{\sigma} ,\psi
 ^{\sigma})$, and as $\psi |F=1$ and $\chi_{|F^*} =1$, one has $\psi
 ^{\sigma}=\psi ^{-1}$ and $\chi ^{\sigma}=\chi ^{-1}$, so that $\gamma(\pi(\chi,
 \chi ^{-\sigma}),\psi)=\gamma(\mu\chi,\psi)\gamma(\mu ^{-1}\chi ^{-1}
 ,\psi ^{-1})=1$.
The conclusion falls from Proposition \ref{gamma1}. \end{proof}

Assuming Theorem 1.2 of \cite{AG}, the converse of Theorem \ref{gamma1} is also true:\\

\begin{thm}
\label{gamma2}

 Let $\pi$ be a smooth irreducible representation of
 infinite dimension of $G_2(K)$ with central character trivial on $F^*$ and $\psi$ a non trivial character of $K/F$, it is distinguished if and only if $\gamma (\pi \otimes \chi ,\psi)=1 $ for every character $\chi$ of $K^*$ trivial on $F^*$.

\end{thm}

\begin{proof} It suffices to show that if $\pi$ is a smooth irreducible distinguished representation of
 infinite dimension of $G_2(K)$, and $\psi$ a non trivial character of $K/F$, then $\gamma(\pi ,\psi)=1$.
Suppose $\lambda$ is a non zero $G_2(F)$-invariant linear form on
 $K(\pi, \psi)$, it is shown in the proof of the corollary of Proposition 3.3
 in \cite{H}, that its restriction to $D(F^*)$ must be a multiple of
 the Haar measure on $F^*$.  
Hence for any function $\phi$ in $D(K^*) \cap \pi(w)D(K^*)$, we must have
$\int_{F^*} \phi(t)d^*t = \int_{F^*} \pi(w)\phi (t)d^*t$.\\
From this one deduces that for any function in $D(K^*) \cap
 \pi(w)D(K^*)$: 
 $$\aligned \int_{K^*}  \pi (w) \phi (x) c_{\pi} ^{-1} (x)   d^* x & = \int_{K^*/F^*} c_{\pi} ^{-1} (a) \int_{F^*}\pi (w) \phi (ta)d^* t da \\
                      &= \int_{K^*/F^*} c_{\pi} ^{-1} (a) \int_{F^*}\pi(\begin{array}{lll} a &
 0 \\ 0& 1 \\ \end{array})\pi (w) \phi (t)d^* t da \\
                      &=\int_{K^*/F^*} c_{\pi} ^{-1} (a) \int_{F^*}\pi (w)
 c_{\pi}(a)\pi(\begin{array}{lll} a^{-1} & 0 \\ 0& 1 \\ \end{array}) \phi (t)d^* t da\\
                      &=  \int_{K^*/F^*}
 \int_{F^*}\pi(\begin{array}{lll} a^{-1} & 0 \\ 0& 1 \\ \end{array}) \phi (t)d^* t da\\
                      &=\int_{K^*/F^*} \int_{F^*}\phi (ta^{-1})d^* t da \\
                      &=\int_{K^*/F^*} \int_{F^*}\phi (ta)d^* t da \\ 
                      &=\int_{K^*} \phi(x) d^* x \endaligned.$$
This implies that either $\gamma(\pi,\psi)$ is equal to one,
 or $\int_{K^*} \phi(x) d^* x$ is equal to zero on $D(K^*) \cap
 \pi(w)D(K^*)$.
In the second case, we could define two independant $K^*$-invariant
 linear forms on $K(\pi, \psi)= D(K^*) + \pi(w)D(K^*)$, given by $\phi_1 +
 \pi(w)\phi_2 \mapsto \int_{K^*} \phi_1(x) d^* x$, and $\phi_1 +
 \pi(w)\phi_2 \mapsto \int_{K^*} \phi_2(x) d^* x$. This would contradict Theorem 1.2 of \cite{AG}.  \end{proof}
 
\section*{Acknowledgements}
I would like to thank Corinne Blondel and Paul G\'{e}rardin, for many helpfull comments and suggestions.

\end{document}